\documentclass[12pt]{amsart}
\usepackage{amsmath, amstext, amsbsy, amssymb}
\usepackage{latexsym,array,delarray,epsfig, color,amsfonts,youngtab,yfonts,inslrmin, textcomp, txfonts, pxfonts}
\usepackage{graphicx, epsfig, textcomp, txfonts, pxfonts}
\usepackage{amsmath,amsthm,amsfonts,amssymb,amsxtra, mathtools} 
\usepackage{comment, cancel}
\usepackage{tikz}

\newtheorem{theorem}{Theorem}[section]

\theoremstyle{definition}

\theoremstyle{remark}

\numberwithin{equation}{section}

\newcommand{\al}{\alpha}

\newcommand{\si}{\sigma}

\newcommand{\Z}{\mathbb Z}
\newcommand{\C}{\mathbb C}
\newcommand{\g}{\mathfrak g}
\newcommand{\h}{\mathfrak h}
\newcommand{\cA}{\mathcal A}
\newcommand{\cK}{\mathcal K}

\def\a{\alpha}

\begin{document}

\title[On realization of some twisted toroidal Lie algebras]
{On realization of some twisted toroidal Lie algebras}
\author{Naihuan Jing, Chad R. Mangum and Kailash C. Misra}
\address{Department of Mathematics,
   North Carolina State University,
   Raleigh, NC 27695-8205, USA}
\email{jing@math.ncsu.edu, misra@math.ncsu.edu}
\address{Department of Mathematics,
   Niagara University,
   5795 Lewiston Road,
   Niagara University, NY 14109, USA}
\email{cmangum@niagara.edu}
\thanks{Jing acknowledges the support of Simons Foundation and NNSF of China,
and Misra acknowledges the support of Simons Foundation.}

\begin{abstract}
Toroidal Lie algebras are generalizations of affine Lie algebras. In 1990, Moody, Rao and Yokonuma gave a presentation for untwisted toroidal Lie algebras. In this paper we give a presentation for the twisted toroidal Lie algebras of type $A$ and $D$ constructed by Fu and Jiang.
\end{abstract}

\maketitle

\section{Introduction} \label{intro}

Since the seventies, infinite dimensional Kac-Moody Lie algebras \cite{K} have had numerous applications in mathematics and physics, mostly thanks to
their subclass of affine Lie algebras $\hat{\mathfrak g}$. Affine Lie algebras consist of two types: the untwisted and twisted ones.
Every automorphism $\sigma$ of the finite dimensional simple Lie algebra $\mathfrak g$ gives rise to a ``twisted'' affine Lie algebra $\hat{\mathfrak g}_{\sigma}$, which is a genuinely twisted affine Lie algebra if and only if $\sigma$ is a Dynkin diagram automorphism of $\mathfrak g$.

In connection with resonance models in physics, Frenkel \cite{F} has constructed loop Kac-Moody Lie algebras using vertex operators
\cite{FLM}. Double affine Lie algebras are special examples of such extended new algebras
beyond affine Lie algebras. The n-toroidal Lie algebras are the universal
central extensions of the n-loop algebras \cite{BK} based on a finite dimensional simple Lie algebra, which have since found
new applications in geometry and physics \cite{S, NSW}.

As a generalization of the untwisted affine Lie algebras $\hat{\mathfrak g}$, Moody, Rao and Yokonuma \cite{MRY} have given a loop algebra realization of the
2-toroidal Lie algebras $T(\mathfrak g)$ with generators and relations.
The MRY realization has led to new constructions of
the toroidal Lie algebras using vertex operators \cite{FM, T1, T2}, other similar techniques used in
the affine Lie algebras \cite{L, JMT}, and
restricted modules and simple modules for extended affine Lie algebras \cite{BB, G}.
Moreover, new realizations of the 2-toroidal Lie algebras of ADE types
in a new form of the McKay correspondence \cite{FJW} were also based upon the MRY realization.

Similar to twisted affine Lie algebras, Fu and Jiang \cite{FJ} have considered twisted $n$-toroidal Lie algebras in an abstract setting
and studied their integrable modules. In \cite{EM} vertex representations of
general toroidal Lie algebras and Virasoro-toroidal Lie algebras have been constructed.

In this paper, we give an MRY-like presentation to the Fu-Jiang twisted 2-toroidal Lie algebras. In recognition of the similarity
between the MRY presentation and Drinfeld realization \cite{D} for quantum affine algebras, we have determined fixed-point subalgebras under a Dynkin diagram automorphism for the toroidal Lie algebras in the $A_{2n+1}, D_{n+1}, D_4$ types and therefore the fixed-point subalgebras
are also realized as central extensions of certain twisted 2-toroidal Lie algebras using field-like
generators. This shows that these twisted toroidal Lie algebras
also enjoy the similar property as twisted affine Lie algebras.

\section{Twisted Toroidal Lie Algebras}

Let $\g$ be the finite dimensional simple Lie algebra $A_{2n-1} , (n \geq 3)$, $D_{n+1}, (n \geq 2)$,
or $D_4$ over the field of complex numbers $\C$. We denote the Chevalley generators of $\g$ by
$\{e_i', f_i', h_i' \mid 1 \leq i \leq N\}$ where $N = 2n - 1, n+1, 4,$ respectively.  Then $\h' = \text{span}\{h_i' \mid 1 \leq i \leq N\}$ is the Cartan
subalgebra of $\g$. Let $\{\al_i' \mid 1 \leq i \leq N\} \subset \h'^*$  denote the simple roots and $\Delta$ be the set of roots for $\g$. Note that
$\al_j'(h_i') = a_{ij}'$ where $A' = (a_{ij}')_{i,j = 1}^N$ is the Cartan matrix associated with $\g$. Let $( \ | \ )$ be the nondegenerate symmetric invariant bilinear form on $\g$ defined by $( x | y ) = tr(xy), \frac{1}{2}tr(xy), \frac{1}{2}tr(xy)$ for all $x, y \in \g$. Then
$( h_i' | h_i' ) = 2 , 1 \leq i \leq N$. Since the Lie algebra $\g$ is simply-laced, we can identify the invariant form on $\h'$ to that on
the dual space $h'^*$ and normalize the inner product by $( \al | \al ) = 2, \al \in \Delta$.

Let $\Gamma$ denote the Dynkin diagram for $\g$ and $\si$ be the following Dynkin diagram automorphism of order $r = 2, 2, 3$ respectively:
\begin{eqnarray*}
&&\sigma(h_i')=h_{N-i+1}', i=1, \cdots , N, \ \ \mbox{for type } \ A_{2n-1} \nonumber \\
&&\sigma(h_i')=h_i', i=1, \cdots , n-1=N-2; \sigma(h_n')=h_{n+1}',
\ \ \mbox{for type} \ D_{n+1} \nonumber\\
&&\sigma(h'_1, h'_2, h'_3, h'_4)=(h'_3, h'_2, h'_4, h'_1)\ \ \mbox{for type} \
D_4.\nonumber
\end{eqnarray*}

Then the Lie algebra $\g$ is decomposed as a ${\Z}/r{\Z}$-graded
Lie algebra:
$${\g}={\g}_0\oplus \cdots\oplus{\g}_{r-1},$$
where ${\g}_i=\{ x\in {\g}| \sigma(x)=\omega^i x\}$ and $
\omega=e^{2\pi\sqrt{-1}/r}$. It is well-known that
the subalgebra ${\g}_0$ is the simple Lie algebra of types
$C_n$, $B_n$ and $G_2$ respectively. Let $I = \{1, 2, \cdots , n\}$ for $\g = A_{2n-1}, D_{n+1}$ and $I = \{1, 2\}$ for $\g = D_4$.
The Chevalley generators $\{e_i, f_i, h_i \mid i \in I\}$ of $\g_0$ are given by:
\begin{eqnarray*}
&&e_i=e_i', f_i=f_i', h_i=h_i', \ \mbox{if } \sigma(i)=i; \\
&&e_i=\sum_{j=0}^{r-1}e'_{\sigma^j(i)}, \
f_i=\sum_{j=0}^{r-1}f'_{\sigma^j(i)}, \
h_i=\sum_{j=0}^{r-1}h'_{\sigma^j(i)}, \ \mbox{if } \sigma(i)\neq i. \\
\end{eqnarray*}

The Cartan subalgebra of $\g_0$ is $\h_0 = \text{span}\{h_i \mid i \in I\}$ and the simple roots $\{\al_i \mid i \in I\}\subset \h_0^*$ are given by:
$$
\al_i = \frac{1}{r}\sum_{s=0}^{r-1}\al'_{\si^s(i)}.
$$


Then we have
\begin{equation}
(\alpha_i|\alpha_j)=d_ia_{ij},
\ \ \mbox{for all} \ i,j\in I
\end{equation}
where $A = (a_{ij})_{i,j \in I}$ is the Cartan matrix for $\g_0$ and $(d_1, \cdots, d_n)=(1/2, \cdots, 1/2, 1)$, $(1, \cdots, 1, 1/2)$,
or $(1/3, 1)$,
for $\g = A_{2n-1}, D_{n+1}$
or  $D_4$ respectively. Then correspondingly $r=2, 2$ and $3$.
Note that $A = (a_{ij})_{i,j \in I}$ is given as follows:
\begin{equation*}
\begin{cases}
a_{ii} = 2, i \in I\\
a_{i,i+1} = a_{i+1,i} = -1, \g_0 \neq G_2\\
a_{12} = -3, a_{21} = -1, \g_0 = G_2\\
a_{n-1,n} = -2, a_{n,n-1} = -1, \g_0 = C_n\\
a_{n-1,n} = -1, a_{n,n-1} = -2, \g_0 = B_n\\
a_{ij} = 0, \text{otherwise}.
\end{cases}
\end{equation*}

Recall that
\begin{eqnarray*}
\theta=\left\{\begin{array}{ll}
\a'_1+\cdots+\a'_{2n-2}+\a'_{2n-1}, \ \mbox{for } A_{2n-1}, \\
\a'_1+2\a'_2+\cdots+2\a'_{n-1}+\a'_n+\a'_{n+1} , \ \mbox{for } D_{n+1}, \\
\a'_1+2\a'_2+\a'_3+\a'_4, \ \mbox{for } D_4,
\end{array}\right.
\end{eqnarray*}
is the highest root in $\g$. Let $f'_0$ denote the $\theta$-root vector, $e'_0$ denote the $(-\theta)$-root vector such that
$[h'_0, e'_0] = 2 e'_0, [h'_0, f'_0] = -2 f'_0$ where $h'_0 = [e'_0 , f'_0]$.

Let $\cA = \C[s, s^{-1}, t, t^{-1}]$ be the ring of Laurent polynomials in the commuting variables $s, t$ and $L(\g) = \g \otimes_{\C} \cA$  be the multi-loop algebra with the Lie bracket given by:

$$
[x \otimes s^jt^m , y \otimes s^kt^l] = [x, y] \otimes s^{j+k}t^{m+l},
$$
for all $x, y \in \g, j, k, m, l \in \Z$. For $j \in \Z$ we define $0 \leq \bar{j} < r$ such that $j \equiv \bar{j} \ \mbox{mod} \ r$. For all $j \in \Z$ we define $\g_j = \g_{\bar{j}}$. We extend the automorphism $\si$ of $\g$ to an automorphism $\bar{\si}$ of $L(\g)$ by defining:
$$
\bar\si(x \otimes s^jt^m) = \omega^{-j} \si(x) \otimes s^jt^m
$$
where $x \in \g, j, m \in \Z$. We denote the $\bar\si$ fixed points of $L(\g)$ by $L(\g, \si)$.
Note that the subalgebra $L(\g, \si)$ has the $\Z$-gradation:
$$
L(\g, \si) = \oplus_{j \in \Z} L(\g, \si)_j,
$$
where $L(\g, \si)_j = \g_j \otimes \cA_j, \cA_j = \mbox{span}_{\C}\{s^jt^m \mid m \in \Z\}=s^j\C[t, t^{-1}]$.

Set $F = \cA \otimes \cA$. Then $F$ is a two sided $\cA$-module via the action $a(b_1 \otimes b_2) = ab_1 \otimes b_2 = (b_1 \otimes b_2)a$ for all $a, b_1, b_2 \in \cA$. Let $G$ be the $\cA$-submodule of $F$ generated by $\{1 \otimes ab -a\otimes b -b \otimes a \mid a, b \in \cA\}$. The $\cA$- quotient module $\Omega_{\cA} = F/G$ is called the $\cA$ - module of K{\"a}hler differentials.  The canonical quotient map $d : \cA \longrightarrow \Omega_{\cA}$ given by $da = (1 \otimes a) + G, a \in \cA$ is the differential map. Let $- : \Omega_{\cA} \longrightarrow
\Omega_{\cA}/d\cA = \cK'$ be the canonical linear map. Since $\overline{d(ab)} = 0$, we have $\overline{a(db)} = - \overline{(da)b} = - \overline{b(da)}$ for all  $a, b \in \cA$. Then $\cK' = span_{\C}\{\overline{bda} \mid a, b \in \cA\}$. Set $\cK = span_{\C}\{\overline{bda} \mid a \in \cA_k, b \in \cA_l , k+l \equiv 0 (mod r)\}$ which is a subalgebra of $\cK'$. We note that $\{\overline{s^{j-1}t^mds}, \overline{s^jt^{-1}dt}, \overline{s^{-1}ds}\mid j \in \Z, m \in \Z_{\neq 0}\}$ is a basis for $\cK$ and the following relations are easy to check.
\begin{equation}\label{kahlercalc}
\overline{s^{\ell}ds^k} = \delta_{k, -\ell} k \overline{s^{-1}ds}, \ \
 \overline{s^{\ell}t^{-1}d(s^{k}t)} = \delta_{k, -\ell} k \overline{s^{-1}ds} + \overline{s^{k + \ell}t^{-1}dt}.
\end{equation}
Let
$$
T(\g) = L(\g , \si) \oplus \cK,
$$
with the Lie bracket given by
\begin{equation*}
[x\otimes a, y\otimes b] = [x, y]\otimes ab + (x | y)\overline{bda}, \qquad [T(\g), \cK] = 0,
\end{equation*}
where $x \in \g_i, y \in \g_j, a \in \cA_i, b \in \cA_j$ for $i, j \in \Z$. Using \cite[Proposition 2.2]{BK}, it is shown \cite[Theorem 2.1]{FJ} that $T(\g)$, with the canonical projection map $\eta : T(\g) \rightarrow L(\g, \si)$, is the universal central extension of $L(\g, \si)$. $T(\g)$ is called the twisted toroidal Lie algebra of type $\g$.

The algebra $T'(\g) = L(\g) \oplus \cK'$ is called the untwisted toroidal Lie algebra of type $\g$. Let $\tilde{A'} = (a_{ij}')_{i,j = 0}^N$ be the affine Cartan matrix of the untwisted affine Lie algebra $\g^{(1)}$.
In \cite{MRY}, Moody, Rao and Yokonuma (MRY) gave a presentation of untwisted toroidal Lie algebra $T'(\g)$ analogous to the Drinfeld realization
\cite{D} of affine Lie algebras as follows.

Let $t'(\g)$ be the algebra generated by $\{{\not c} , a'_i(k) , X(\pm a'_i, k) \mid 0 \leq i \leq N, k \in \Z\}$ satisfying the relations:
\begin{equation}\label{compuntw}
\begin{cases}
[\not{c} , a'_i (k)] = [\not{c} , X(\pm a'_i , k)] = 0\\
 [a'_i (k) , a'_j (m)] = (\a'_i | \a'_j) \delta_{k , -m} k \not{c}\\
 [a'_i (k) , X(\pm a'_j , m)] = \pm (\a'_i | \a'_j) X(\pm a'_j , m+k)\\
 \displaystyle [X(a'_i , m) , X(-a'_j , k)] = -\delta_{ij} \left( a'_i (m+k) + \frac{2m\delta_{m , -k}}{(\a'_i | \a'_j)} \not{c} \right)\\
 [X(\pm a'_i , m) , X(\pm a'_i , k)] = 0\\
 \left( \text{ad} X(\pm a'_i , m)\right)^{1-a'_{ij}} X(\pm a'_j , k) = 0, \ \text{for}  \  i \neq j ,\\
 \end{cases}
\end{equation}
where $i, j \in \{0, 1, \cdots , N\}$ and $k,m \in \Z$.

 Let $z,w,z_1,z_2,...$ be formal variables. We define formal power series with
 coefficients from the toroidal Lie algebra $t'(\g)$:
 $$
 a'_i(z)=\sum_{n\in \Z}a'_i(n)z^{-n-1},
\qquad X(\pm a'_i,z)=\sum_{n\in \Z}X(\pm a'_i, n)z^{-n-1},
 $$
 for $i=0,1, \cdots , N$.
We will use the delta function
\begin{equation*}
\delta(z-w)=\sum_{n\in\Z}w^nz^{-n-1}
\end{equation*}

Using $\displaystyle \frac1{z-w}=\sum_{n=0}^{\infty}z^{-n-1}w^n$, $|z|>|w|$,
we have the following useful expansions:
\begin{align*}
\delta(z-w)&=\iota_{z, w}((z-w)^{-1})+\iota_{w, z}((w-z)^{-1}),\\
\partial_w\delta(z-w)&=
\iota_{z, w}((z-w)^{-2})-\iota_{w, z}((w-z)^{-2}),
\end{align*}
where $\iota_{z, w}$ means the expansion at the region $|z|>|w|$. For simplicity
in the following we will drop $\iota_{z, w}$ if it is
clear from the context.

Now we can write the relations in (\ref{compuntw}) in terms of power series as follows:

\begin{equation}\label{seriesuntw}
\begin{cases}
[\not{c} , a'_i (z)] = [\not{c} , X(\pm a'_i , z)] = 0\\
 [a'_i (z) , a'_j (w)] = (\a'_i | \a'_j) \partial_w\delta(z - w) \not{c}\\
 [a'_i (z) , X(\pm a'_j , w)] = \pm (\a'_i | \a'_j) X(\pm a'_j , w)\delta(z - w)\\
 [X(a'_i , z) , X(-a'_j , w)] = -\delta_{ij} \frac{2}{(\a_i | \a_j)}\left( a'_i (z)\delta(z - w) + \partial_w\delta(z - w) \not{c} \right)\\
 [X(\pm a'_i , z) , X(\pm a'_i , w)] = 0\\
 \text{ad} X(\pm a'_i , z_2) X(a'_j , z_1) = 0, \ \text{if}  \  a'_{ij} = 0, i \neq j \\
 \text{ad} X(\pm a'_i , z_3)\text{ad} X(\pm a'_i , z_2) X(a'_j , z_1) = 0, \ \text{if}  \  a'_{ij} = -1, i \neq j \\
\text{ad} X(\pm a'_i , z_4) \text{ad} X(\pm a'_i , z_3)\text{ad} X(\pm a'_i , z_2) X(a'_j , z_1) = 0,  \text{if} \  a'_{ij} = -2, i \neq j ,
\end{cases}
\end{equation}
where $i, j \in \{0, 1, \cdots , N\}.$

Define the map $\pi : t'(\g) \longrightarrow L(\g)$ by:
\begin{equation*}
\begin{cases}
 \not{c} \mapsto 0,\\
 a'_j (k) \mapsto h'_{j} \otimes s^k, \\
 X(a'_0, k) \mapsto e'_0 \otimes s^kt, \\
 X(-a'_0, k) \mapsto - f'_0 \otimes s^kt^{-1}, \\
 X(a'_i, k) \mapsto e'_{i} \otimes s^k, \\
 X(-a'_i, k) \mapsto - f'_{i} \otimes s^k,
 \end{cases}
\end{equation*}
for $0 \leq j \leq N$, $1 \leq i \leq N$.

The following theorem shows that $t'(\g)$ gives a realization of the untwisted toroidal Lie algebra $T'(\g)$.

\begin{theorem} (\cite[Proposition 3.5]{MRY})\label{untwisted}:
The map $\pi$ is a surjective homomorphism, the kernel of $\pi$ is contained in the center $Z(t'(\g))$ and $(t'(\g), \pi)$ is the universal central extension of $L(\g)$.
\end{theorem}

\section{MRY presentation of $T(\g)$}

In this section we give an MRY type presentation for the twisted toroidal Lie algebra $T(\g)$ which is the main result in this paper. For
$\g = A_{2n-1}, r=2, 1 \leq i \leq n-1$ we define:

\begin{equation}\label{twgenA}
\begin{cases}
 a_0(k) = a'_0 (k) \ \text{for} \  k \in 2\Z, \\
 a_i(k) = a'_i (k) + (-1)^{k}\si (a'_i) (k) \ \text{for} \  k \in \Z,\\
 a_n(k) = 2 a'_n (k) \ \text{for} \  k \in 2\Z, \\
 X(\pm a_0, k) = X(\pm a'_0, k) \  \text{for} \  k \in 2\Z, \\
 X(\pm a_i, k) = X(\pm a'_i, k) + (-1)^{k}X(\pm \si(a'_i), k) \ \text{for} \  k \in \Z, \text{and} \\
 X(\pm a_n, k) = 2 X(\pm a'_n, k) \  \text{for} \  k \in 2\Z.
\end{cases}
\end{equation}

For  $\g = D_{n +1}, r=2, 1 \leq i \leq n-1$ we define:

\begin{equation}\label{twgenD}
\begin{cases}
 a_0(k) = a'_0 (k) \ \text{for} \  k \in 2\Z, \\
 a_i(k) = 2a'_i (k) \ \text{for} \  k \in 2\Z,\\
 a_n(k) =  a'_n (k) + (- 1)^ka'_{n+1}(k) \ \text{for} \  k \in \Z, \\
 X(\pm a_0, k) = X(\pm a'_0, k) \  \text{for} \  k \in 2\Z, \\
 X(\pm a_i, k) = 2X(\pm a'_i, k) \  \text{for}  \  k \in 2\Z, \\
 X(\pm a_n, k) = X(\pm a'_n, k) + (-1)^kX(\pm a'_{n+1}, k) \  \text{for} \  k \in \Z.
\end{cases}
\end{equation}

For  $\g = D_4, r =3$ we define:
\begin{equation}\label{twgenD4}
\begin{cases}
a_0(k) = a'_0 (k) \ \text{for} \ k \in 3\Z \\
 a_1(k) = a'_1 (k) + \omega^{-k}a'_3(k) + \omega^{-2k}a'_4(k) \  \text{for} \  k \in \Z\\
 a_2(k) = 3 a'_2 (k) \  \text{for} \  k \in 3\Z \\
 X(\pm a_0, k) = X(\pm a'_0, k)  \ \text{for} \ k \in 3\Z \\
 X(\pm a_1, k) = X(\pm a'_1, k) + \omega^{-k} X(\pm a'_3, k) \\ + \omega^{-2k} X(\pm a'_4, k) ,\ \text{for} \  k \in \Z \\
X(\pm a_2, k) = 3 X(\pm a'_2, k) \ \text{for} \ k \in 3\Z.
\end{cases}
\end{equation}

Denote $\tilde{I} = I \cup \{0\}$ and extend the Cartan matrix $A = (a_{ij})_{i,j \in I}$ to $\tilde{A} = (a_{ij})_{i,j \in \tilde{I}}$ by defining $a_{00} = 2$, $a_{01} = -1 = a_{10} \ \text{for} \ \g = A_{2n-1}$, $a_{02} = -1 = a_{20} \ \text{for} \ \g = D_{n+1} \ \text{or} \ D_4$, and $a_{0j} = 0 = a_{j0}$ otherwise.

Let $t(\g)$ be the subalgebra of $t'(\g)$ generated by $\{{\not c} , a_0(2k) , a_i(k), a_n(2k), \\ X(\pm a_0, 2k), X(\pm a_i, k), X(\pm a_n, 2k) \mid 1 \leq i \leq n-1, k \in \Z\}$, $\{{\not c} , a_0(2k) , a_i(2k), \\ a_n(k), X(\pm a_0, 2k), X(\pm a_i, 2k), X(\pm a_n, k) \mid 1 \leq i \leq n-1, k \in \Z\}$, and $\{{\not c} , a_0(3k) , a_1(k), a_2(3k), X(\pm a_0, 3k), X(\pm a_1, k), X(\pm a_2, 3k) \mid k \in \Z\}$ for $\g = A_{2n-1}, D_{n+1}$ and $D_4$ respectively. Then using the untwisted relations (\ref{compuntw}) we obtain the following relations among the generators of $t(\g)$ (with $\not{c}$ being central). Below we assume  $k, k_i, l \in \Z, 2\Z \ \text{or} \ 3\Z$ depending on the generators of $t(\g)$ as above.
\begin{enumerate}\label{twrelations}
  \item $[a_0(k), a_0(l)] =
      a_{00} k \delta_{k,- l} \not{c} $ \\
  \item $[a_0(k), a_j(l)] =
      r a_{0j} k \delta_{k,- l} \not{c} $ \\
where $1 \leq j \leq n$.
  \item $[a_i(k), a_j(l)] =
    \begin{cases}
      r a_{ij} k \delta_{k,-\ell} \not{c} & (A_{2n-1}) \\
      r^2 a_{ij} k \delta_{k,-\ell} \not{c} & (D_{n+1}) \\
      r a_{ij} \delta_{k,0 \text{mod}r} k \delta_{k,-\ell} \not{c} & (D_4)
    \end{cases}$ \\
where $1 \leq i \leq j \leq n$ and $(i, j) \neq (n-1, n), (n, n)$.
 \item $[a_{n-1}(k), a_n(l)] =
    \begin{cases}
      r a_{n-1,n} k \delta_{k,-\ell} \not{c} & (A_{2n-1}, D_4) \\
      r^2 a_{n-1,n} k \delta_{k,-\ell} \not{c} & (D_{n+1})
    \end{cases}$ \\
 \item $[a_n(k), a_n(l)] =
    \begin{cases}
      r^2 a_{nn} k \delta_{k,-\ell} \not{c} & (A_{2n-1}, D_4) \\
      r a_{nn} k \delta_{k,-\ell} \not{c} & (D_{n+1}) \\
    \end{cases}$
 \item $[a_0(k), X(\pm a_j, l)] =
      \pm a_{0j} X(\pm a_j, k+l) $ \\
where $0 \leq j \leq n$.
 \item $[a_i(k), X(\pm a_0, l)] =
    \begin{cases}
      \pm r \delta_{k,0 \text{mod}r} a_{i0} X(\pm a_0, k+l) & (A_{2n-1}, D_{n+1}) \\
      \pm r a_{i0} X(\pm a_0, k+l)  & (D_4)
    \end{cases}$ \\
where $1 \leq i \leq n$.
 \item $[a_i(k), X(\pm a_j, l)] =
    \begin{cases}
      \pm a_{ij} X(\pm a_j, k+l) & (A_{2n-1}, D_4) \\
      \pm r a_{ij} X(\pm a_j, k+l) & (D_{n+1}) \\
    \end{cases}$ \\
where $1 \leq i,j \leq n$ and $(i, j) \neq (n-1, n), (n, n-1), (n, n)$.
 \item $[a_{n-1}(k), X(\pm a_n, l)] =
    \begin{cases}
      \pm \delta_{k,0 \text{mod}r} a_{n-1,n} X(\pm a_n, k+l) & (A_{2n-1}, D_4) \\
      \pm r a_{n-1,n} X(\pm a_n, k+l)& (D_{n+1})
    \end{cases}$ \\
 \item $[a_n(k), X(\pm a_{n-1}, l)] =
    \begin{cases}
      \pm r a_{n,n-1} X(\pm a_{n-1}, k+l) & (A_{2n-1}, D_4) \\
      \pm \delta_{k,0 \text{mod}r} a_{n,n-1} X(\pm a_{n-1}, k+l) & (D_{n+1})
    \end{cases}$ \\
 \item $[a_n(k), X(\pm a_n, l)] =
    \begin{cases}
      \pm r a_{nn} X(\pm a_n, k+l) & (A_{2n-1}, D_4) \\
      \pm a_{nn} X(\pm a_n, k+l) & (D_{n+1})
    \end{cases}$
 \item $[X(\pm a_i, k) , X(\pm a_i, l) ] = 0 $ \\
where $0 \leq i \leq n$.
 \item $[X(a_i, k) , X(-a_j, l) ] =
    \begin{cases}
      -\delta_{i,j} \big\{ a_i(k+l) ( 1 + \delta_{i,n}) \\ + k \delta_{k,-\ell} ( 1 + (1 - \delta_{i,0}) + 2 \delta_{i,n} ) \not{c} \big \} & (A_{2n-1}) \\
      -\delta_{i,j} \big\{ a_i(k+l) \big( 1 + (1 - \delta_{i,0} - \delta_{i,n}) \big) \\ + k \delta_{k,-\ell} \big( 4 - 3 \delta_{i,0} - 2 \delta_{i,n} \big) \not{c} \big \} & (D_{n+1}) \\
      -\delta_{i,j} \big\{ a_i(k+l) ( 1 + 2 \delta_{i,2}) \\ + k \delta_{k,-\ell} ( 1 + 2 \delta_{i,1} + 8 \delta_{i,2} ) \not{c} \big \} & (D_4)
    \end{cases}$  \\
where $0 \leq i,j \leq n$
 \item $\text{ad} X(\pm a_p , k_2) X(\pm a_m , k_1) = 0$ for $0 \leq p \neq m \leq n$ and $a_{pm} = 0$
 \item $\text{ad} X(\pm a_p, k_3) \text{ad} X(\pm a_p , k_2) X(\pm a_m , k_1) = 0$ for $0 \leq p \neq m \leq n$ and $a_{pm} = -1$
 \item $\text{ad} X(\pm a_p, k_4) \text{ad} X(\pm a_p , k_3) \text{ad} X(\pm a_p , k_2) X(\pm a_m , k_1) = 0$ for $0 \leq p \neq m \leq n$ and $a_{pm} = -2$
 \item $\text{ad} X(\pm a_p, k_5) \text{ad} X(\pm a_p, k_4) \text{ad} X(\pm a_p , k_3) \text{ad} X(\pm a_p , k_2) X(\pm a_m , k_1) = 0$ for $0 \leq p \neq m \leq n$ and $a_{pm} = -3$
\end{enumerate}

The above defining relations of $t(\g)$ can be written in terms of power series as follows.
\begin{enumerate}
  \item $[a_0(z), a_0(w)] =
      \frac{1}{r} a_{00} \sum_{j=0}^{r-1} \partial_w\delta(z-\omega^{-j} w) \not{c}$
  \item $[a_0(z), a_j(w)] =
      a_{0j} \sum_{j=0}^{r-1} \partial_w\delta(z-\omega^{-j} w) \not{c}$ \\
where $1 \leq j \leq n$.
 \item $[a_i(z), a_j(w)] =
    \begin{cases}
      r a_{ij} \partial_w\delta(z-w) \not{c} & (A_{2n-1}) \\
      r a_{ij} \sum_{j=0}^{r-1} \partial_w\delta(z-\omega^{-j} w) \not{c} & (D_{n+1})  \\
      a_{ij} \sum_{j=0}^{r-1} \partial_w\delta(z-\omega^{-j} w) \not{c} & (D_4)
    \end{cases}$ \\
where $1 \leq i \leq j \leq n$ and $(i, j) \neq (n-1, n), (n, n)$.
 \item $[a_{n-1}(z), a_n(w)] =
    \begin{cases}
      a_{n-1,n} \sum_{j=0}^{r-1} \partial_w\delta(z-\omega^{-j} w) \not{c} & (A_{2n-1}, D_4) \\
      r a_{n-1,n} \sum_{j=0}^{r-1} \partial_w\delta(z-\omega^{-j} w) \not{c} & (D_{n+1})
    \end{cases}$
 \item $[a_n(z), a_n(w)] =
    \begin{cases}
      r a_{nn} \sum_{j=0}^{r-1} \partial_w\delta(z-\omega^{-j} w) \not{c} & (A_{2n-1}, D_4) \\
      r a_{nn} \partial_w\delta(z-w) \not{c} & (D_{n+1})
    \end{cases}$
 \item $[a_0(z), X(\pm a_j, w)] =
      \pm \frac{1}{r} a_{0j} X(\pm a_j, w) \sum_{j=0}^{r-1} \delta(z-\omega^{-j} w)$ \\
where $0 \leq j \leq n$.
 \item $[a_i(z), X(\pm a_0, w)] =
      \pm a_{i0} X(\pm a_0, w) \sum_{j=0}^{r-1} \delta(z-\omega^{-j} w)$ \\
where $1 \leq i \leq n$.
 \item $[a_i(z), X(\pm a_j, w)] =
    \begin{cases}
      \pm a_{ij} X(\pm a_j, w) \delta(z-w) & (A_{2n-1}, D_4)  \\
      \pm a_{ij} X(\pm a_j, w) \sum_{j=0}^{r-1} \delta(z-\omega^{-j} w) & (D_{n+1})
    \end{cases}$ \\
where $1 \leq i,j \leq n$ and $(i, j) \neq (n-1, n), (n, n-1), (n, n)$.
 \item $[a_{n-1}(z), X(\pm a_n, w)] =
    \begin{cases}
      \pm \frac{1}{r} a_{n-1,n} X(\pm a_n, w) \sum_{j=0}^{r-1} \delta(z-\omega^{-j} w) & (A_{2n-1}, D_4) \\
      \pm a_{n-1,n} X(\pm a_n, w) \sum_{j=0}^{r-1} \delta(z-\omega^{-j} w) & (D_{n+1})
    \end{cases}$
 \item $[a_n(z), X(\pm a_{n-1}, w)] =
    \begin{cases}
      \pm a_{n,n-1} X(\pm a_{n-1}, w) \sum_{j=0}^{r-1} \delta(z-\omega^{-j} w) & (A_{2n-1}, D_4) \\
      \pm \frac{1}{r} a_{n,n-1} X(\pm a_{n-1}, w) \sum_{j=0}^{r-1} \delta(z-\omega^{-j} w) & (D_{n+1})
    \end{cases}$
 \item $[a_n(z), X(\pm a_n, w)] =
    \begin{cases}
      \pm a_{nn} X(\pm a_n, w) \sum_{j=0}^{r-1} \delta(z-\omega^{-j} w) & (A_{2n-1}, D_4) \\
      \pm a_{nn} X(\pm a_n, w) \delta(z-w) & (D_{n+1}) \\
    \end{cases}$
 \item $[X(\pm a_i, z), X(\pm a_i, w)] = 0$ \\
where $0 \leq i \leq n$.
 \item $[X(a_i, z), X(- a_j, w)] = \\
    \begin{cases}
      -\delta_{i,j} \big\{ ( 1 - \frac{1}{2} \delta_{i,0}) a_i(w) \big( \delta(z-w) + (\delta_{i,0} + \delta_{i,n}) \delta(z+w) \big) + \\
\indent ( 2 - \frac{3}{2} \delta_{i,0} ) \big(\partial_w\delta(z-w) + (\delta_{i,0} + \delta_{i,n}) \partial_w\delta(z+w) \big) \not{c} \big \} & (A_{2n-1}) \\
      -\delta_{i,j} \big\{ ( 1 - \frac{1}{2} \delta_{i,0}) a_i(w) \big( \delta(z-w) + (1 - \delta_{i,n}) \delta(z+w) \big) + \\
\indent ( 2 - \frac{3}{2} \delta_{i,0} ) \big(\partial_w\delta(z-w) + (1 - \delta_{i,n}) \partial_w\delta(z+w) \big) \not{c} \big \} & (D_{n+1}) \\
      -\delta_{i,j} \big\{ ( 1 - \frac{2}{3} \delta_{i,0}) a_i(w) \big( \delta(z-w) + \delta(z-\omega^{-1} w) + \delta(z-\omega^{-2} w) \\ \indent - \delta_{i,1} (\delta(z-\omega^{-1} w) + \delta(z-\omega^{-2} w)) \big) + \\
\indent ( 3 - \frac{8}{3} \delta_{i,0} ) \big( \partial_w\delta(z-w) + \partial_w\delta(z-\omega^{-1} w) + \partial_w\delta(z-\omega^{-2} w) \\ \indent - \delta_{i,1} (\partial_w\delta(z-\omega^{-1} w) + \partial_w\delta(z-\omega^{-2} w)) \big) \not{c} \big \} & (D_4)
    \end{cases}$  \\
where $0 \leq i,j \leq n$.
 \item $\text{ad} X(\pm a_p , z_2) X(\pm a_m , z_1) = 0$ for $0 \leq p \neq m \leq n$ and $a_{pm} = 0$
 \item $\text{ad} X(\pm a_p, z_3) \text{ad} X(\pm a_p , z_2) X(\pm a_m , z_1) = 0$ for $0 \leq p \neq m \leq n$ and $a_{pm} = -1$
 \item $\text{ad} X(\pm a_p, z_4) \text{ad} X(\pm a_p , z_3) \text{ad} X(\pm a_p , z_2) X(\pm a_m , z_1) = 0$ for $0 \leq p \neq m \leq n$ and $a_{pm} = -2$
 \item $\text{ad} X(\pm a_p, z_5) \text{ad} X(\pm a_p, z_4) \text{ad} X(\pm a_p , z_3) \text{ad} X(\pm a_p , z_2) X(\pm a_m , z_1) = 0$ for $0 \leq p \neq m \leq n$ and $a_{pm} = -3$
\end{enumerate}

By definition of $L(\g, \si)$ and $t(\g)$ it is clear that $\pi (t(\g)) \subseteq L(\g, \si)$. Hence $\bar{\pi} = \pi |_{t(\g)} : t(\g) \longrightarrow L(\g, \si)$. Explicitly the map $\bar{\pi}$ is given by:
\begin{equation*}
\begin{cases}
 \not{c} \mapsto 0, \\
 a_0(k) \mapsto h'_0 \otimes s^k \\
 a_i(k) \mapsto \sum_{j=0}^{r-1} \sigma^{j} (h'_{i}) \otimes (\omega^{-j} s)^k \\
 X( a_0, k) \mapsto e'_0 \otimes s^kt \\
 X(- a_0, k) \mapsto -f'_0 \otimes s^kt^{-1} \\
 X(a_i, k) \mapsto \sum_{j=0}^{r-1} \sigma^{j} (e'_{i}) \otimes (\omega^{-j} s)^k \\
 X(-a_i, k) \mapsto \sum_{j=0}^{r-1} \sigma^{j} (f'_{i}) \otimes (\omega^{-j} s)^k
\end{cases}
\end{equation*}
for $1 \leq i \leq n$. The following theorem which is the main result of this paper shows that $t(\g)$ gives a realization of the twisted toroidal Lie algebra $T(\g)$.

\begin{theorem}
The map $\bar{\pi}$ is a surjective homomorphism, the kernel of $\bar{\pi}$ is contained in the center $Z(t(\g))$ and $(t(\g), \bar{\pi})$ is the universal central extension of $L(\g, \si)$.
\end{theorem}

\begin{proof} Using \cite[Proposition 8.3]{K}, it follows that the map $\bar{\pi}$ is surjective. By Theorem \ref{untwisted}, we have $ker(\pi) \subseteq Z(t'(\g))$. Hence $ker(\bar{\pi}) = ker(\pi) \cap t(\g) \subseteq Z(t'(\g)) \cap t(\g) = Z(t(\g))$. Therefore, $(t(\g), \bar{\pi})$ is a central extension of $L(\g, \si)$. It is left to show that $(t(\g), \bar{\pi})$ is the universal central extension of $L(\g, \si)$. Suppose $(\mathcal{V} , \gamma)$ is a central extension of $L(\g , \si)$. Since $(T(\g), \eta)$ is the universal central extension of $L(\g , \si)$,  we have a unique map $\lambda : T(\g) \rightarrow \mathcal{V}$ such that $\gamma \lambda = \eta$. Define the map $\psi : t(\g) \rightarrow T(\g)$ by:

\begin{equation}\label{psi}
\begin{cases}
 \not{c} \mapsto \overline{s^{-1}ds} \\
 a_0(k) \mapsto h'_0 \otimes s^k + \overline{s^{k}t^{-1}dt} \\
 a_i(k) \mapsto \sum_{j=0}^{r-1} \sigma^{j} (h'_{i}) \otimes (\omega^{-j} s)^k \\
 X( a_0, k) \mapsto e'_0 \otimes s^kt \\
 X(- a_0, k) \mapsto -f'_0 \otimes s^kt^{-1} \\
 X( a_i, k) \mapsto \sum_{j=0}^{r-1} \sigma^{j} (e'_{i}) \otimes (\omega^{-j} s)^k \\
 X(- a_i, k) \mapsto -\sum_{j=0}^{r-1} \sigma^{j} (f'_{i}) \otimes (\omega^{-j} s)^k
\end{cases}
\end{equation}
for $1 \leq i \leq n$. Note that the map $\psi$ and $\bar{\pi}$ differ only on $\not{c}$ and $a_0(k)$ by elements of $\cK$. Hence
$\eta \psi = \bar{\pi}$. Since $\bar{\pi}$ is a homomorphism, to show that $\psi$ is a homomorphism it suffices to show that $\psi$ preserves the defining relations involving $a_0(k)$ by direct calculations. For example, using (\ref{kahlercalc}) when $\g = A_{2n-1}$ or $D_{n+1}$ we have:\\

$\displaystyle \big[ \psi \big( a_0(k) \big ), \psi \big( a_j(l) \big) \big]
= \big[ h'_0 \otimes s^k + \overline{s^{k}t^{-1}dt} ,
 h'_{j} \otimes s^l + \sigma (h'_j) \otimes (-s)^l \big]$ \\
$\displaystyle = ( h'_0 | h'_j ) \overline{s^{l}ds^k} + ( h'_0 | \sigma (h'_j) ) \overline{s^{l}ds^k} (-1)^{l}
= r a_{0j} k \delta_{k,-l}  \psi(\not{c})$ \\
since for $\g = A_{2n-1}$, $( h'_0 | h'_j ) = ( h'_0 | \sigma (h'_j) ) = -\delta_{j,1} = a_{0j}$ and for $\g =D_{n+1}$, $( h'_0 | h'_j ) = ( h'_0 | \sigma (h'_j) ) = -\delta_{j,2} = a_{0j}$. \\

Similarly, when $\g = D_4$ using (\ref{kahlercalc}) we have:

$\displaystyle \big[ \psi \big( a_0(k) \big ), \psi \big( a_j(l) \big) \big]\\
= \bigg[ h'_0 \otimes s^k + \overline{s^{k}t^{-1}dt} ,
 h'_j \otimes s^{l} + \sigma (h'_j) \otimes (\omega^{-1} s)^{l} + \sigma^2 (h'_j) \otimes (\omega^{-2} s)^{l} \bigg]$ \\
$\displaystyle = \big( ( h'_0 | h'_j ) + ( h'_0 | \sigma (h'_j) ) (\omega^{-1})^{k+l} + ( h'_0 | \sigma^{2} (h'_j) ) (\omega^{-2})^{k+l} \big) \overline{s^{l}ds^{k}} \\
\displaystyle = a_{0j} \big( 1 + (\omega^{-1})^{k+l} + (\omega^{-2})^{k+l} \big ) \overline{s^{l}ds^{k}}= ra_{0j}  k \delta_{k,-l} \psi (\not{c})$, \\
since $k \in 3\Z$ and $( h'_0 | h'_j ) = ( h'_0 | \sigma (h'_j) ) = ( h'_0 | \sigma^2 (h'_j) ) = -\delta_{j,2} = a_{0j}$.

Thus we have a homomorphism $\lambda\psi : t(\g) \longrightarrow \mathcal{V}$ and
$\gamma\lambda\psi = \eta\psi = \bar{\pi}$ giving the following commuting diagram:

\begin{center}
\begin{tikzpicture}[node distance=2cm, auto]
  \node (tauh) {$T(\g)$};
  \node (drinfeldt) [below of=tauh] {$t(\g)$};
  \node (V) [above of=tauh] {$\mathcal{V}$};
  \node (L) [right of=tauh] {$L(\g , \si)$};
  \draw[->] (tauh) to node {$\lambda$} (V);
  \draw[->] (tauh) to node {$\eta$} (L);
  \draw[->] (V) to node {$\gamma$} (L);
  \draw[->] (drinfeldt) to node {$\psi$} (tauh);
  \draw[->] (drinfeldt) to node [swap] {$\bar{\pi}$} (L);
\end{tikzpicture}
\end{center}

Since $T(\g)$ is the universal central extension of $L(\g, \si)$, the lower triangle commutes which implies that the map $\psi$ is unique.

\end{proof}

\bibliographystyle{amsalpha}

\end{document}